\documentclass[11pt]{amsart}

\usepackage{amssymb}
\usepackage{epsfig}
\usepackage{mathrsfs}

\def\RR{\mathbb{R}}

\def\PV{\mathrm {PV}\,}
\def\supp{\mathrm {supp}\,}

\def\vphi{\varphi}

\def\eps{\varepsilon}

\def\ds{\displaystyle}

\newtheorem{theorem}{Theorem}[section]
\newtheorem{lemma}[theorem]{Lemma}
\newtheorem{proposition}[theorem]{Proposition}

\def\qed{\hbox{${\vcenter{\vbox{                        
   \hrule height 0.4pt\hbox{\vrule width 0.4pt height 6pt
   \kern5pt\vrule width 0.4pt}\hrule height 0.4pt}}}$}}

\begin{document}

\title[Non local fronts]{Existence and asymptotics of fronts in non local combustion models}
\author{Antoine Mellet, Jean-Michel Roquejoffre and Yannick Sire}
\date{}
\maketitle

\begin{abstract}
We prove the existence and provide the asymptotics for non local fronts in homogeneous media.  
\end{abstract}

\tableofcontents

\section{Introduction}

This paper is devoted to the study of fronts propagation in homogeneous media for a fractional reaction-diffusion equation appearing in combustion theory.
More precisely, we consider the following classical scalar model for the combustion of premixed gas 
with ignition temperature:
\begin{equation} \label{eq:1}
u_t + (-\partial_{xx})^\alpha u = f(u) \quad  \mbox{ in } \RR\times \RR,
\end{equation}
where the function $f$ satisfies:
\begin{equation}\label{eq:beta}
\left\{
\begin{array}{l}
f:\RR\rightarrow \RR \mbox{ continuous function} \\[5pt]
f(u) \geq 0 \mbox{ for all } u\in\RR\mbox{ and }\supp f = [\theta,1] \\[5pt]
f'(1)<0 
\end{array}
\right.
\end{equation}
where $\theta\in(0,1)$ is a fixed number (usually referred to as the ignition temperature).
\medskip

The operator $(-\partial_{xx})^\alpha$ denotes the fractional power of the Laplace operator in one dimension (with $\alpha\in(0,1]$). It can be defined by the following singular
integral 
\begin{equation}\label{eq:lapa}
(-\partial_{xx})^\alpha u(x)=c_{\alpha}\, \PV \int_{\mathbb{R}} \frac{u(x)-u(z)}{|x-z|^{1+2\alpha}}\,dz
\end{equation}
where $\PV$ stands for the Cauchy principal
value. 
This integral is well defined, for instance, if $u$ belongs to $C^2(\mathbb R)$ and satisfies
$$\int_\mathbb R \frac{|u(x)|}{(1+|x|)^{1+2\alpha}}\,dx < +\infty $$
(in particular, smooth bounded functions are admissible). 
Alternatively, the fractional Laplace operator can be defined as a
pseudo-differential operator with symbol $|\xi|^{2\alpha}$. We
refer the reader to the book by Landkof where an extensive study of
$(-\partial_{xx})^\alpha$ is performed by means of harmonic analysis techniques (see \cite{landkof}). 
\vspace{10pt}

In this paper, we will always take $\alpha \in (1/2, 1]$, and we are interested in particular solutions of (\ref{eq:1}) which describe transition fronts between the stationary states $0$ and $1$ (traveling fronts).
These traveling fronts are solutions of (\ref{eq:1}) that are of the form
\begin{equation}\label{eq:TW}
u(t,x) = \phi(x+ct)
\end{equation}
with
$$\left\{
\begin{array}{l}
\ds \lim_{x\rightarrow-\infty}\phi(x) = 0 \\[3pt]
\ds \lim_{x\rightarrow+\infty}\phi(x) = 1 .
\end{array}
\right.
$$
The number $c$ is the speed of propagation of the front. It is readily seen that $\phi$ must solve
$$
 (-\partial_{xx})^\alpha \phi + c \, \phi' =f(\phi) \qquad \mbox{ for all }x\in\RR
$$

\vspace{10pt}

When $\alpha=1$ (standard Laplace operator), it is well known that there exists a unique speed $c$ and a unique profile $\phi$ (up to translation) that correspond to a traveling front solution of (\ref{eq:1}) (see e.g. \cite{BLL,BN,BNS}). 
The goal of this paper is to generalize these results to the case $\alpha\in (1/2, 1)$.
We are thus looking for  $\phi$ and $c$ satisfying
\begin{equation} \label{1DwaveNL}
\left\{
\begin{array}{l}
 (-\partial_{xx})^\alpha \phi + c\,  \phi' =f(\phi) \qquad \mbox{ for all }x\in\RR\\[3pt]
\ds \lim_{x\rightarrow-\infty}\phi(x) = 0 \\[3pt]
\ds \lim_{x\rightarrow+\infty}\phi(x) = 1 \\[3pt]
\phi(0)=\theta
\end{array}
\right.
\end{equation}
(the last condition is a normalization condition which ensures the uniqueness of $\phi$).
Our main theorem is the following:
\begin{theorem}\label{thm:exist}
Let $\alpha \in (1/2,1)$ and assume that $f$ satisfies (\ref{eq:beta}), then
there exists a unique pair  $(\phi_0,c_0)$ solution of \eqref{1DwaveNL}. 
Furthermore, $c_0>0$ and  $\phi_0$ is monotone increasing.
\end{theorem}   

We will also obtain the following result, which describes the asymptotic behavior of the front at $-\infty$:
\begin{theorem}\label{thm:asymp}
Let $\alpha \in (1/2,1)$ and assume that $f$ satisfies (\ref{eq:beta}). Let $\phi_0$ be the unique solution of \eqref{1DwaveNL} provided by Theorem \ref{thm:exist}. 
Then there exist $m,M$ such that
$$  \phi_0(x)\leq \frac{M}{|x|^{2\alpha-1}}\quad \mbox{ for } x\leq -1$$
and 
$$ \phi_0'(x) \geq \frac{m}{|x|^{2\alpha}} \qquad \mbox{ for } x\leq -1.$$
\end{theorem}

The proof of Theorem \ref{thm:exist} follows classical arguments developed by Berestycki-Larrouturou-Lions \cite{BLL} (see also Berestycki-Nirenberg \cite{BN}): 
Truncation of the domain, construction of sub- and super-solutions  and passage to the limit. As usual, one of the main difficulty is to make sure that we recover a finite, non trivial speed of propagation at the limit.
The main novelty (compared with similar results when $\alpha=1$) is the construction of sub- and super-solutions where the classical exponential profile is replaced by power tail functions.

 \vspace{20pt}

\section{Truncation of the domain}\label{1Dwave}
The first step is to truncate the domain: for some $b>0$, we consider 
the following problem:
 \begin{equation} \label{truncatedNL}
\left\{
\begin{array}{l}
(-\partial_{xx})^\alpha \phi_b + c_b \phi'_b =f(\phi_b)\qquad \mbox{ for all }x\in [-b,b]\\[3pt]
\phi_b (x) = 0 \qquad \mbox{for $s$} \leq -b\\[3pt]
\phi_b(x) = 1 \qquad \mbox{for $s$ } \geq b \\[3pt]
\phi_b(0)=\theta.
\end{array}
\right.
\end{equation}
The goal of this section is to prove that this problem has a solution for $b$  large enough. More precisely, we are now going to prove: 
\begin{proposition}\label{prop:truncate}
Assume $\alpha \in (1/2,1)$ and that $f$ satisfies (\ref{eq:beta}). 
Then there exists a constant $M$ such that if $b>M$  the truncated problem \eqref{truncatedNL} has a unique solution ($\phi_b$, $c_b$). Furthermore,  the following properties hold:
\begin{itemize}
\item[(i)] There exists $K$ independent of $b$ such that $-K\leq c_b\leq K$.
\item[(ii)] $\phi_b$ is non-decreasing with respect to $x$ and satisfies $0 < \phi_b(x) <1$ for all $x\in(-b,b)$.
\end{itemize} 
\end{proposition} 
Before we can prove this Proposition, we need to detail the construction of sub- and super-solutions.

\subsection{Construction of sub- and super-solutions}
In the proof of the existence of traveling waves for the standard Laplace operator ($\alpha=1$), sub- and super-solution of the form $e^{\gamma x}$ play a crucial role,  in particular in the determination of the asymptotic behavior of the traveling waves as $x\rightarrow -\infty$. 
These particular functions are replaced, in the case of the fractional Laplace operator, by functions with polynomial tail.
In what follows, we will rely on two important lemmas:
\begin{lemma}\label{lem:phi}
Let $\beta\in(0,1)$ and 
define
$$ \vphi(x)=\left\{
\begin{array}{ll}
\frac{1}{|x|^{\beta}} & \mbox{ if } x<-1 \\
1 & \mbox{ if }x>-1.
\end{array}
\right.
$$ 
Then $\vphi$ satisfies
$$(-\partial_{xx})^\alpha \vphi + c \vphi'(x) = \frac{-c_{\alpha}}{2\alpha |x|^{2\alpha}}  + c  \frac{\beta}{|x|^{\beta+1}} + O\left(\frac{1}{|x|^{\beta+2\alpha}}\right)$$
when $x\rightarrow -\infty$.
\end{lemma}
and 
\begin{lemma}\label{lem:phi2}
Let $\beta>1$ and 
define
$$ \bar \vphi(x)=\left\{
\begin{array}{ll}
\frac{1}{|x|^{\beta}} & x<-1 \\
0 & x>-1
\end{array}
\right.
$$ 
then
$$(-\partial_{xx})^\alpha \bar \vphi+ c \bar \vphi'(x) =\frac{-c_{\alpha}}{\beta-1}\frac{1}{|x|^{2\alpha+1}}  + c  \frac{\beta}{|x|^{\beta+1}} + O\left(\frac{1}{|x|^{\beta+2\alpha}}\right)$$
when $x\rightarrow -\infty$.
\end{lemma}

\begin{proof}[Proof of Lemma \ref{lem:phi}]
We want to estimate  $(-\partial_{xx})^\alpha \vphi$ for $x<-1$. We have:
$$(-\partial_{xx})^\alpha \vphi(x)=- c_{\alpha} \PV \int_{\RR} \frac{\vphi(x+y)-\vphi(x)}{|y|^{1+2\alpha}}\,dy,$$
which we decompose as follow:
\begin{eqnarray*}
(-\partial_{xx})^\alpha \vphi(x)& = & c_{\alpha}\int_{-\infty}^{-1-x}\frac{\vphi(x)-\vphi(x+y)}{|y|^{1+2\alpha}}\,dy+c_{\alpha}\int_{-1-x}^{+\infty} \frac{\vphi(x)-\vphi(x+y)}{|y|^{1+2\alpha}}\,dy\\
& = & I+II
\end{eqnarray*}

A simple explicit computation yields: 
$$II=\left(\frac{1}{|x|^\beta}-1\right) \frac{c_{\alpha}}{2\alpha |x+1|^{2\alpha}}.$$
Performing the change of variables $y=xz$, one gets 
\begin{eqnarray*}
I & = &  \frac{c_{\alpha}}{|x|^{\beta+2\alpha}}\int_{+\infty}^{-\frac{1}{x}-1} \frac{|z+1|^\beta-1}{|z+1|^\beta |z|^{1+2\alpha}}\, dz.
\end{eqnarray*}
Note that  the integrand has a singularity at $z=0$, and this integral has to be understood as a principal value. 
We also observe  that the integrand  has a singularity at $z=-1$, but since $\beta<1$, this singularity is integrable, and thus
$$ I\sim -c_{\alpha}\frac{1}{|x|^{\beta+2\alpha}} \PV \int_{-1}^{+\infty} \frac{|z+1|^\beta-1}{|z+1|^\beta |z|^{1+2\alpha}}\, dz.\qquad \mbox{ as } x\rightarrow -\infty .$$ 
We deduce:
$$(-\partial_{xx})^\alpha \vphi(x)=\frac{-c_{\alpha}}{2\alpha |x|^{2\alpha}} +O\left(\frac{1}{|x|^{\beta+2\alpha}}\right)$$
when $x\rightarrow -\infty$, and the result follows.
\end{proof}

\bigskip

\begin{proof}[Proof of Lemma \ref{lem:phi2}]
Again, we decompose $(-\partial_{xx})^\alpha \bar \vphi$  as follow:
\begin{eqnarray*}
(-\partial_{xx})^\alpha \bar \vphi(x)& = & c_{\alpha}\int_{-\infty}^{-1-x}\frac{\bar \vphi(x)-\bar \vphi(x+y)}{|y|^{1+2\alpha}}\,dy+c_{\alpha}\int_{-1-x}^{+\infty} \frac{\bar \vphi(x)-\bar \vphi(x+y)}{|y|^{1+2\alpha}}\,dy\\
& = & I+II
\end{eqnarray*}

Now, a simple explicit computation yields: 
$$II=\frac{c_{\alpha}}{|x|^\beta} \frac{1}{2\alpha |x+1|^{2\alpha}}.$$
And performing the change of variables $y=xz$, one gets 
\begin{eqnarray*}
I & = &  \frac{c_{\alpha}}{|x|^{\beta+2\alpha}}\int_{+\infty}^{-\frac{1}{x}-1} \frac{|z+1|^\beta-1}{|z+1|^\beta |z|^{1+2\alpha}}\, dz.
\end{eqnarray*}
Note that  the integrand as a singularity at $z=0$, and this integral has to be understood as a principal value. 
We also observe  that the integrand  has a singularity at $z=-1$ and since $\beta>1$, this singularity is divergent and thus 
$$ I\sim \frac{-c_{\alpha}}{\beta-1} |x|^{\beta-1}. $$
We deduce:
$$(-\partial_{xx})^\alpha \bar \vphi(x)=\frac{-c_{\alpha}}{\beta-1}\frac{1}{|x|^{2\alpha+1}} +O\left(\frac{1}{|x|^{\beta+2\alpha}}\right)$$
which yields the result.
\end{proof}

\bigskip

\subsection{Proof of Proposition \ref{prop:truncate}}
We now turn to the proof of  Proposition \ref{prop:truncate}.
First, we fix $c \in \RR$ and consider the following problem: 
 \begin{equation} \label{eq:phic}
\left\{
\begin{array}{l}
(-\partial_{xx})^\alpha \phi+ c\, \phi' =f(\phi)\qquad \mbox{ for all }x\in [-b,b]\\[3pt]
\phi (x) = 0 \qquad \mbox{for } x \leq -b\\[5pt]
\phi(x) = 1 \qquad \mbox{for } x \geq b 
\end{array}
\right.
\end{equation}
We have:
\begin{lemma}
For any $c\in \RR$, Equation (\ref{eq:phic}) has a unique solution $\phi_c$. Furthermore $\phi_c$ is non-decreasing with respect to $x$ and $c \to \phi_c$ is continuous. 
\end{lemma}
\begin{proof}
Since $1$ and $0$ are respectively super- and sub-solutions, we can use Perron's method (recall that the fractional laplacian enjoys a comparison principle) to prove the existence of a solution $\phi_c(x)$ for any $c\in \RR$. 
By a sliding argument, we can show that $\phi_c$ is unique and non-decreasing with respect to $x$. The fact that the function $c \to \phi_c$ is continuous follows from classical arguments (see \cite{BN} for  details).
\end{proof}

We now have to show that there exists a unique $c=c_b$ such that $\phi_{c_b}(0)=\theta$. 
This will be a consequence of the following lemma:
\begin{lemma}\label{lem:c}
There exist  constants $M$, $K$  such that  for $b>M$ the followings hold:
\begin{enumerate}
 \item if $c>K$ then the solution of (\ref{eq:phic}) satisfies $\phi_c(0)<\theta$,
 \item if $c<-K$  then the solution of (\ref{eq:phic}) satisfies $\phi_c(0)>\theta$.
\end{enumerate}
\end{lemma}

Together with the fact that $\phi_{c}(0)$ is continuous with respect to $c$, Lemma~\ref{lem:c} implies that  there exists $c_b\in[-K,-K]$ such that $\phi_{c_b}$ satisfies  $\phi_{c_b}(0)=\theta$ and is thus a solution of  (\ref{truncatedNL}). This completes the proof of Proposition \ref{prop:truncate}.

\begin{proof}[Proof of Lemma \ref{lem:c}] 
We consider the function
\begin{equation}\label{eq:phi2} 
\vphi(x)=\left\{
\begin{array}{ll}
\frac{1}{|x|^{2\alpha -1}} & x<-1 \\
1 & x\geq -1
\end{array}
\right.
\end{equation}
and note that Lemma \ref{lem:phi} (with $\beta=2\alpha-1$) yields that if $c$ is large enough ($c\geq\frac{c_{\alpha}}{2\alpha(2\alpha-1)}$),  then 
$$ (-\partial_{xx})^\alpha \vphi(x) + c\vphi '(x)  \geq 0$$
for $x\leq - A$ (for some $A$ large enough). We can also assume that $\vphi(x)\leq \theta$ for $x\leq -A$, and so
$$  (-\partial_{xx})^\alpha \vphi(x) + c \vphi'(x)  \geq f(\vphi)=0\qquad \mbox{ for } x\leq- A.$$
Furthermore, for $-A<x< -1$, $ (-\partial_{xx})^\alpha \vphi(x) $ is bounded while
$$ c\vphi'(x) \geq c  \frac{2\alpha-1}{A^{2\alpha}}.$$
For $c$ large enough, we thus have
$$ (-\partial_{xx})^\alpha \vphi(x) + c\vphi'(x)  \geq \sup f\geq f(\vphi)\qquad\mbox{ for $-A<x<-1$.}$$
We deduce that there exists $K$ such that if  $c\geq K$ then
$$ (-\partial_{xx})^\alpha \vphi(x) + c\vphi'(x)  \geq f(\vphi)\qquad\mbox{ for $x<-1$}$$
and so $\vphi$ is a supersolution for (\ref{eq:phic}).

Choosing $M$ such that $\vphi(-M)<\theta$, we now see that if  $c\geq K$ and  $b>M$, then $\vphi(x-M)$ is a super-solution for (\ref{eq:phic}).
By a sliding argument, we deduce that $\phi_c(x) \leq \vphi(x-M)$ and so  $\phi_c(0)\leq \vphi(-M)<\theta$.

\bigskip

For the lower bound, we define $\vphi_1(x) = 1-\vphi(-x)$. Then we we have, if $-c\geq K$ ($c\leq -K$) and for $x >1$
$$  (-\partial_{xx})^\alpha\vphi_1(x) + c \vphi_1'(x) = -[(-\partial_{xx})^\alpha\vphi(-x) + (-c)\vphi'(-x)] \leq 0\leq f(\vphi). $$
Moreover,  we have $\vphi_1(x)=0$ for $x\leq 1$.
Proceeding as above, we deduce that if $c \leq -K $, then $\phi_c(0) > \theta$, which concludes the proof.
\end{proof}

\vspace{20pt}

\section{Proof of Theorem \ref{thm:exist}}
In order to complete  the proof of  Theorem \ref{thm:exist}, we have to prove that we can pass to the limit $b\rightarrow\infty$ in the truncated problem.
More precisely, Theorem \ref{thm:exist} follows from the following proposition:
\begin{proposition}\label{prop:limit}
Under the conditions of Proposition \ref{prop:truncate},
there exists a subsequence $b_n\rightarrow \infty$ such that $\phi_{b_n}\longrightarrow \phi_0$ and $c_{b_n} \longrightarrow c_0$. Furthermore,  $c_0 \in (0,K]$ and  $\phi_0$ is a monotone increasing solution of (\ref{1DwaveNL}).
\end{proposition}
\begin{proof}[Proof of Proposition \ref{prop:limit}]

We recall that $ c_b \in [-K,K]$, and classical elliptic estimates (see \cite{BCP}) yield: 
$$ || \phi_b||_{\mathcal C^{2,\gamma}} \leq C$$
for some $\gamma \in (0,1)$. 
Thus there exists a subsequence $b_n\rightarrow \infty$ such that 
$$ c_n :=c_{b_n} \longrightarrow c_0\in[-K,K]$$
$$ \phi_n := \phi_{b_n}  \longrightarrow \phi_0$$
as $n\rightarrow \infty$. 
It is readily seen that $\phi_0$ solves
\begin{equation}\label{eq:9}
 (-\partial_{xx})^\alpha \phi_0 + c_0\, \phi_0' =f(\phi_0)\qquad \mbox{ for all }x\in \RR.
\end{equation}

It is also readily seen that $\phi_0(x)$ is monotone increasing, $\phi_0(0)=\theta$ and $\phi_0$ is bounded. By a standard compactness argument, there exists $\gamma_0$, $\gamma_1$  such that $\lim_{x\rightarrow -\infty}\phi_0(x) = \gamma_0$ and   $\lim_{x\rightarrow +\infty}\phi_0(x) = \gamma_1$ with
$$0\leq \gamma_0\leq \theta\leq\gamma_1\leq 1.$$
It remains to prove that $c_0>0$,  $\gamma_0=0$ and  $\gamma_1=1$.
For that, we will mainly follow classical  arguments (see \cite{BLL}, \cite{BH}).

\vspace{15pt}
First, we have the following lemma:
\begin{lemma}\label{lem:int0}
The function $\phi_0$ satisfies
$$\int_\RR (-\partial_{xx})^\alpha \phi_0 (x) \,dx=0. $$
\end{lemma}
\begin{proof}[Proof of Lemma \ref{lem:int0}]
The result follows formally by integrating formula (\ref{eq:lapa}) with respect to $x$ and using the antisymmetry with respect to the variables $x$ and $z$. However, because of the principal value, one has to be a little bit careful with the use of Fubini's theorem.

To avoid this difficulty, we will use instead the equivalent formula for the fractional laplacian:
\begin{eqnarray} 
 (-\partial_{xx})^\alpha \phi_0 (x) &  =  & c_\alpha \int_{\RR\setminus [x-\eps,x+\eps]} \frac{ \phi_0(x)- \phi_0(z)}{|x-z|^{1+2\alpha}}\, dz\nonumber \\
 && +  c_\alpha \int_{[x-\eps,x+\eps]} \frac{ \phi_0(x)-\phi_0(z)+\phi_0'(x)(z-x)}{|x-z|^{1+2\alpha}}\, dz\label{eq:fracs}
  \end{eqnarray}
which is valid for all $\eps>0$ and does not involve singular integrals.
Integrating the first term with respect to $x\in \RR$, and using Fubini's theorem, we get
\begin{eqnarray*}
\int_\RR  \int_{\RR\setminus [x-\eps,x+\eps]} \frac{ \phi_0(x)- \phi_0(z)}{|x-z|^{1+2\alpha}}\, dz \, dx
& = & \int_\RR  \int_{\RR\setminus [z-\eps,z+\eps]} \frac{ \phi_0(x)- \phi_0(z)}{|x-z|^{1+2\alpha}}\, dx \, dz\\
& =& - \int_\RR  \int_{\RR\setminus [x-\eps,x+\eps]} \frac{ \phi_0(x)- \phi_0(z)}{|x-z|^{1+2\alpha}}\, dz \, dx
\end{eqnarray*}
and so this integral vanishes.
Using Taylor's theorem, the second term  in (\ref{eq:fracs}) can be rewritten as
\begin{eqnarray*}
\int_{x-\eps}^{x+\eps} \frac{1}{|x-z|^{1+2\alpha}} \int_x^z (z-t) \phi_0''(t)\, dt\, dz 
&= & 
\int_{-\eps}^{\eps} \frac{1}{|y|^{1+2\alpha}} \int_x^{x+y}  (y+x-t) \phi_0''(t)\, dt\, dy .
\end{eqnarray*}
Integrating with respect to $x$ and using (twice) Fubini's theorem, we deduce
\begin{eqnarray*}
&  & \int_\RR \int_{x-\eps}^{x+\eps} \frac{1}{|x-z|^{1+2\alpha}} \int_x^z (z-t) \phi_0''(t)\, dt\, dz \, dx \\
& &\qquad\qquad =
\int_{-\eps}^{\eps} \frac{1}{|y|^{1+2\alpha}}\int_{-\infty}^{+\infty} \int_x^{x+y}  (y+x-t) \phi_0''(t)\, dt\, dx\, dy\\
& &\qquad\qquad =
\int_{-\eps}^{\eps} \frac{1}{|y|^{1+2\alpha}}\int_{-\infty}^{+\infty} \int_{t-y}^{t}  (y+x-t) \phi_0''(t)\, dx\, dt\, dy\\
& &\qquad\qquad =
\int_{-\eps}^{\eps} \frac{y^2}{2 |y|^{1+2\alpha}}\int_{-\infty}^{+\infty}  \phi_0''(t) \, dt\, dy\\
& &\qquad\qquad = 0,
\end{eqnarray*}
where we used the fact that $\lim_{x\to \pm \infty } \phi_0'(x) = 0$ and so $\int_{-\infty}^{+\infty}  \phi_0''(t) \, dt=0$.
The lemma follows.
\end{proof}

Now, we can  integrate equation (\ref{eq:9}) with respect to $x\in \RR$, and using Lemma \ref{lem:int0}, we get: 
\begin{equation}\label{eq:intbeta}
\int_{\RR} f(\phi_0(x))\, dx  = c_0 (\gamma_1-\gamma_0)<\infty.
\end{equation}
In particular, we observe that (\ref{eq:intbeta}) implies that 
$$ f(\gamma_0 ) = f(\gamma_1)=0,$$
otherwise the integral would be infinite. 
\vspace{15pt}

Next, we prove:
\begin{lemma} \label{lem:speed}
The limiting speed satisfies:
$$c_0>0.$$
\end{lemma}
\begin{proof}
First of all, we note that for all $n$, there exists $a_n\in(0,b_n)$ such that $\phi_n(a_n) =\frac{1+\theta}{2}$. Furthermore, up to another subsequence, by elliptic estimates, the function $\psi_n(x)= \phi_{b_n}(a_n+x)$ converges to a function $\psi_0$. 
Note that since $\psi_0\in \mathcal C^\gamma$,  there exists $r>0$  such that
$$ \psi_0(x) \in \left[\frac{3+\theta}{4}, \frac{1+3\theta}{4}\right] \quad \mbox{ for } x\in[-r,r]$$
and so there exists $\kappa_0>0$ such that
\begin{equation}\label{eq:beta0}
\int_\RR f (\psi_0)\, dx >\kappa_0. 
\end{equation}

Up to a subsequence, we can assume that $b_n+a_n$ is either convergent or goes to $+\infty$.
We need to distinguish the two cases:
\medskip

\noindent {\bf Case 1: $b_n+a_n\rightarrow+ \infty$:} In that case, $\psi_0$ solves
\begin{equation}\label{eq:psi0}
(-\partial_{xx})^\alpha \psi_0+ c_0 \psi_0' =f(\psi_0)\qquad \mbox{ for all }x\in \RR.
\end{equation}
Furthermore, $\psi_0(0)=\frac{1+\theta}{2}$ and $\psi_0$ is monotone increasing. In particular, it is readily seen that there exists $\bar \gamma_0$ and $\bar \gamma_1$ such that   $\lim_{x\rightarrow -\infty}\psi_0(x) =\bar \gamma_0$ and   $\lim_{x\rightarrow +\infty}\psi_0(x) =\bar \gamma_1$ with
$$0\leq \bar \gamma_0\leq\frac{1+\theta}{2}\leq\bar \gamma_1\leq 1.$$
Integrating (\ref{eq:psi0}) over $\RR$, and using the fact that
$$
\int_\RR (-\partial_{xx})^\alpha \psi_0(x)\, dx =0
$$
(the proof is the same as in Lemma \ref{lem:int0})
we deduce
\begin{equation}\label{eq:c0g} c_0 (\bar\gamma_1-\bar\gamma_0) = \int_\RR f(\psi_0)\, dx<\infty
\end{equation}
and so 
$$ f(\bar  \gamma_0) = f (\bar  \gamma_1) = 0.$$
This implies that
$$ \bar \gamma_1=1\qquad \mbox{ and }\qquad \bar \gamma_0 \leq \theta.$$
Finally, (\ref{eq:c0g}) and (\ref{eq:beta0}) yields
$$ c_0 (1-\theta) \geq  \int_\RR f(\psi_0)\, dx\geq \kappa_0$$
which gives the result.
\medskip

\noindent {\bf Case 2: $a_n+b_n\rightarrow \bar a<\infty$:}  
 In that case, $\psi_0$ solves
\begin{equation}\label{eq:psi02}
(-\partial_{xx})^\alpha \psi_0+ c_0 \psi_0' =f(\psi_0)\qquad \mbox{ for all }x\in (-\infty,\bar a)
\end{equation}
and we need to modify the proof slightly. 
First, we notice that $\psi_0(x)=1$ for $x\geq  \bar a$, and we observe that $(-\partial_{xx})^\alpha\psi_0(x)\geq 0$ for $x\geq  \bar a$. In particular
$$
\int_{-\infty}^{\bar a}(-\partial_{xx})^\alpha\psi_0(x)\, dx \leq  \int_\RR (-\partial_{xx})^\alpha\psi_0(x)\, dx =0
$$
Proceeding as above, we check that $\lim_{x\rightarrow -\infty}\psi_0(x) =\bar \gamma_0\leq \theta$ and
integrating (\ref{eq:psi02}) over $(-\infty,\bar a)$,   we deduce
$$ c_0(1- \theta) \geq  \int_\RR f(\psi_0)\, dx>0.$$ 

\end{proof}

\bigskip

The positivity of the speed, together with the sub-solution constructed in Lemma \ref{lem:phi} will now give $\gamma_0=0$. More precisely, we now prove:
\begin{lemma} \label{eq:lemlim} The function $\phi_0$ satisfies:
$$\lim_{x\rightarrow -\infty} \phi_0(x) = 0.$$
\end{lemma}
\begin{proof}
Let $c_1=c_0/2>0$ and take $n$ large enough so that  $c_{b_n}\geq c_1$.

We recall that by Lemma \ref{lem:phi}  (see also the  proof of Lemma \ref{lem:c}) that the function
$$
\vphi(x)=\left\{
\begin{array}{ll}
\frac{1}{|x|^{2\alpha -1}} & x<-1 \\
1 & x>-1
\end{array}
\right.
$$
satisfies 
$$ (-\partial_{xx})^\alpha\vphi + K\vphi'  \geq 0 \qquad \mbox{ in } \{\vphi<1\}$$
for some $K$ large enough.
Introducing $\vphi_\eps(x) =\vphi(\eps x)$, we deduce
$$(-\partial_{xx})^\alpha\vphi_\eps + \eps^{2\alpha-1} K \vphi_\eps'(x)  \geq 0 \qquad \mbox{ in } \{\vphi_\eps(x)<1\}$$
and taking $\eps$ small enough (recalling that $2\alpha>1$), we get
$$ (-\partial_{xx})^\alpha\vphi_\eps + c_1 \vphi_\eps'(x)  \geq 0 \qquad \mbox{ in } \{\vphi_\eps<1\}.$$
Furthermore, $\vphi_\eps=1$ for $x\geq 0$, and so  by a sliding argument, we deduce $\phi_{b_n}(x) \leq \vphi_\eps(x)$ for all $n$ such that $c_{b_n}\geq c_1$ and thus
$$ \phi_0(x) \leq \vphi_\eps(x)$$
which implies in particular that $\gamma_0=0$.
\end{proof}

\bigskip

Finally, we conclude the proof of Proposition \ref{prop:limit} by proving that $\gamma_1=1$:
\begin{lemma} \label{eq:lemlim2} The function $\phi_0$ satisfies:
$$\lim_{x\rightarrow +\infty} \phi_0(x) = 1$$
\end{lemma}
\begin{proof}
We recall that (\ref{eq:intbeta}) implies that either $\gamma_1=\theta$ or $\gamma_1=1$ (otherwise the integral is infinite).
Furthermore, if $\gamma_1=\theta$, then $\phi_0\leq \theta$ on $\RR$ and so $\int_\RR f(\phi_0(x))\, dx =0$. Since  $\gamma_0 = 0 <\theta$, (\ref{eq:intbeta}) implies $c_0=0$, which is a contradiction. Hence $\gamma_1 =1$.
\end{proof}

\end{proof}

\medskip

\section{Asymptotic behavior}
We now prove Theorem \ref{thm:asymp}, which further characterizes the behavior of $\phi_0$ as $x\rightarrow -\infty$. We recall that in the case of the regular Laplacian ($\alpha=1$), $\phi_0$ and its derivatives decrease exponentially fast to $0$ as $x\rightarrow -\infty$. When $\alpha\in(1/2,1)$, it is readily seen that the  proof of Lemma \ref{eq:lemlim} actually implies: 
\begin{proposition}[Asymptotic behavior of $\phi_0$]\label{prop:asymp1}
There exists  $M$ such that
$$  \phi_0(x)\leq \frac{M}{|x|^{2\alpha-1}}\quad \mbox{ for } x\leq -1$$
\end{proposition}

\medskip

Noticing that $\phi'_0>0$ solves
$$
(-\partial_{xx})^\alpha\phi_0''+c_0(\phi_0')' = 0 \quad \mbox{ for } x\leq 0,
$$
we can also prove:
\begin{proposition}[Asymptotic behavior of $\phi_0'$]\label{prop:asymp2}
There exists a constant $m$ such that
$$ \phi_0'(x) \geq \frac{m}{|x|^{2\alpha}} \qquad \mbox{ for } x\leq -1.$$
\end{proposition}
\begin{proof}
Lemma \ref{lem:phi2} implies that the function 
$$ \bar \vphi(x)=\left\{
\begin{array}{ll}
\frac{1}{|x|^{2\alpha}} & x<-1 \\
0 & x>-1
\end{array}
\right.
$$ 
satisfies
$$(-\partial_{xx})^\alpha\bar \vphi + c \bar \vphi'(x) =-\frac{c_{\alpha}}{2\alpha-1}\frac{1}{|x|^{2\alpha+1}}  + c  \frac{2\alpha}{|x|^{2\alpha+1}} + O\left(\frac{1}{|x|^{4\alpha}}\right)$$
when $x\rightarrow \infty$,
and so
$$(-\partial_{xx})^\alpha\bar \vphi + k \bar \vphi'(x) \leq 0 \quad \mbox{ for } x\leq - A$$
if $k$ is small enough and $A$ is large.

We introduce $\vphi_\eps(x)=\bar \vphi(\eps x)$, which satisfies
$$(-\partial_{xx})^\alpha\vphi_\eps + \eps^{1-2\alpha}k\vphi_\eps'  \leq 0 \quad\mbox{ for } x< - \eps^{-1}A $$
hence
$$(-\partial_{xx})^\alpha\vphi_\eps + c_0\vphi_\eps'  \leq 0 \quad\mbox{ for } x<-\eps^{-1}A $$
provided we choose $\eps$ small enough.

Finally, we take $r$ so that
$$ \phi_0'(x) \geq r\vphi_\eps(x) \qquad \mbox{ for } -\eps^{-1}A<x<-\eps^{-1}.$$
Proposition \ref{prop:asymp2} now follows from the maximum principle and a sliding argument using the fact that $\vphi_\eps(x)=0$ for $x\geq -\eps^{-1}$.
\end{proof}

\bigskip


\bibliographystyle{alpha}
\bibliography{biblio}

\end{document}